\documentclass[letter,10pt]{amsart}
\usepackage{amsmath,amssymb,amsfonts,amsxtra}
\usepackage{graphicx,caption,subcaption}
\title [Flux Recovery on Quadratic Immersed FEM]{ Flux Recovery and Superconvergence of Quadratic Immersed Interface Finite Elements }
\author{So-Hsiang Chou$^\dagger$ and C. Attanayake$^\ddagger$}
\thanks{$\dagger$
 Department of Mathematics and Statistics, Bowling Green
State University, Bowling Green, OH, 43403-0221. {{\tt email:
chou@bgsu.edu}};\,$\ddagger$ Department of Mathematics,
Miami University
Middletown, OH 45042. {{\tt
e-mail:attanac@muohio.edu}} }
\date{Preprint submitted to IJNAM on 10-23-2015}						

\newtheorem{theorem}{Theorem}[section]

\newcommand{\al}{\alpha}
\def \xj {x_{j,1}}
\def \xjh {x_{j,2}}
\def \xjo {x_{j,3}}
\def \xja {x_{\alpha/2}}

\begin{document}
\maketitle
\begin{abstract}
We introduce a flux recovery scheme for the computed solution of a quadratic  immersed finite element method introduced by Lin {\it et al.} in \cite{LinSun}. The recovery is done at nodes and
interface point first and by interpolation at the remaining points. We show that the end nodes are superconvergence points for both the primary variable $p$ and its flux $u$.
Furthermore, in the case of  piecewise constant diffusion coefficient without the absorption term  the errors at end nodes and interface point in the approximation of  $u$ and $p$ are zero.
In the general case,  flux error at end nodes and  interface point is third order.  Numerical results are provided to
confirm the theory.
\end{abstract}

\section{Introduction}
We consider the interface two-point boundary value problem
\begin{equation}
\begin{cases}
-(\beta(x)p'(x))'+q(x)p(x)= f(x), &  x\in (a,b),  \label{eqn:First}\\
p(a)= p(b) = 0,
\end{cases}
\end{equation}
where $q(x)\ge 0$ and $0<\beta\in C(a,\alpha)\cup C(\alpha,b))$ is  piecewise constant  with a finite
 jump across the interface point $\alpha$ so that the solution $p$ satisfies
\begin{eqnarray}
[p]_{\alpha} =0,  \label{jump1}\\
\left[\beta p'\right]_{\alpha}=0,  \label{jump2}
\end{eqnarray}
where
$[s]_{\alpha}=s^+-s^-$ denotes the jump of the quantity $s$ across $\alpha$.

Physically the variable $p$ may stand for  the pressure or temperature in a material with certain physical properties
and the derived quantity $u:=-\beta p'$ is the corresponding flux, which may be of equal interest. The piecewise constant $\beta$ reflects a nonuniform material and the function
$q(x)$ reflects a property of the material or its surroundings. In this paper we will refer to $p$ as pressure. Problem (\ref{eqn:First})-(\ref{jump2}) can also be viewed as the steady neutron diffusion problem \cite{Stak}. Due to
its simple structure, a lot of its mathematical and numerical properties of related numerical methods can be explicitly worked out.
Therefore, it is very instructive to study this problem before moving to its higher dimensional and/or nonsteady versions. It is in this sprit that we shall study the immersed finite elements for this problem. Efficient numerical methods for (\ref{eqn:First})-(\ref{jump2})
may  use  meshes that are either fitted or unfitted with the interface. A method allowing unfitted meshes would be very efficient when one has to follow a moving interface in a temporal problem.  For an in-depth exposition of the numerics and applications of interface problems, we refer the reader to \cite{LI:2006} and the references therein. In an immersed finite element (IFE) method, the mesh is made up of interface elements where the interface intersects elements  (thus immersed) and noninterface elements where the interface is absent. On a noninterface element one uses standard local shape functions, whereas on an interface element one uses piecewise standard local shape functions subject to continuity and jump conditions. Representative works
on IFE methods can be found in \cite{He, LI, LI:2006, Li2004, Li2003, TLin2001, Zhan}, among others.
We are interested in studying IFE methods that can produce accurate approximate flux  $u_h$ of $p$ once an approximate $p_h$ has been obtained,
 particularly those that can recover flux without having to solve a system of equations. Chou and Tang \cite{ChouTang} initiated such methods when the mesh is
fitted. Later it was generalized to the immersed interface mesh case using linear immersed finite elements (IFE) of Lin {\it et al.} \cite{LinSun}
  and their variants for one dimensional elliptic and parabolic problems \cite{AttanChou, Chou}. In this paper we concentrate on quadratic elements. We aim at a method that
will extend good features such as existence of superconvergence points, discrete conservation law that we
have  either proved or observed in the linear case.

To begin with, let's first give the central idea \cite{ChouTang} behind our flux recovery scheme on a mesh $\{t_i\}$. Suppose we want to evaluate $u(t_i)$ at some mesh point $t_i$ using some weighted
integral of $p$. We can proceed as follows. Let $\phi$ be a function with compact support $K$ such that $I_i=[t_{i-1},t_{i}]\subset K$, the interface point $\alpha\not\in K$, $\phi(t_{i-1})=0$, $\phi(t_i)=1$. An example of such a function is the standard finite element hat function.  Multiplying (\ref{eqn:First}) by $\phi$ and integrating by parts, we see that the flux $u$ satisfies
\[   u(t_i)=-\int_{I_i}\beta p'\phi'dx-\int_{I_i}qp\phi dx+\int_{I_i}f\phi dx.\]
It is then natural to define an approximate flux $u_h$ at $t_i$ as

\[   u_h(t_i)=-\int_{I_i}\beta p_h'\phi'dx-\int_{I_i}qp_h\phi dx+\int_{I_i}f\phi dx.\]
The error $E_i:=u(t_i)-u_h(t_i)$ then satisfies
\[   E_i=-\int_{I_i}\beta (p'-p_h')\phi'dx-\int_{I_i}q(p-p_h)\phi dx.\]
In the case that $\phi$ is linear on $I_i$, $q=0$, $p=p_h$ at $t_{i-1},t_i$, we immediately see that the error in flux is also zero at $t_i$.
With a little calculation using the jump conditions (\ref{jump1})-(\ref{jump2}), the same line of thought works when $\alpha\in I_i$.
In this paper the $\phi$'s will be from the immersed quadratic shape functions and we show in Thm \ref{thm4.3} that in the case of $q=0$, the quadratic IFE solution $p_h=p$ at all end nodes and as a consequence $u=u_h$ at those points as well. When $q\not=0$, the exactness cannot be attained due to the nature of the Green's
function involved (see the proof Thm \ref{thm4.4}), but those points are still superconvergence points of the pressure and flux .
 Another feature of our scheme is that when $q=0$ the following conservation law or discrete first fundamental theorem
of calculus holds:
\[    u_h(t_i)-u_h(t_{i-1})=\int_{I_i}f(x)dx,\]
whose continuous version can be obtained for the exact flux from integrating (\ref{eqn:First}).
The above two features in higher dimensional IFE methods are under investigation \cite{Chou2}. Finally, since the IFE reduces to the standard finite elements in absence of the interface, the superconvergence results in this paper also apply to
the standard finite elements and are consistent with those corresponding results in \cite{Wahlbin} when applicable.
The organization of this paper is as follows.  In Section 2 we introduce the quadratic immersed finite element space of  Lin {\it et al.} \cite{LinSun}  and its approximation properties.
In Section three we give the main pointwise error estimates for both pressure and flux. In the last section we provide numerical
results to confirm the theory.

\section{Approximation Space}\label{AppSpc}

%
Consider the weak formulation of the interface problem (\ref{eqn:First})-(\ref{jump2}) : Find $p\in H_0^{1}(a,b)$ such that
\begin{equation}\label{weak}
\int_{a}^{b}\beta(x)p'(x)v'(x)dx+\int_a^b q(x)p(x)v(x)dx = \int_{a}^{b}f(x)v(x)dx \quad \forall v\in H_0^{1}(a,b),
\end{equation}
where $f\in L^{2}(a,b)$. It is known that the solution $p\in H_0^{1}(a,b)$ exists and  further $p\in H^{2}(a,\alpha)\cap H^{2}(\alpha,b)$.
We use functions in the quadratic IFE space introduced in \cite{LinSun} to approximate $p$.
 Let $a=t_{0}<t_{1}<\ldots<t_{k}<t_{k+1}\ldots <t_{n}=b$ be a partition of $I=[a,b]$, and the interface point $\alpha\in (t_{k},t_{k+1})$ for some $k$. Let $h=\max_{1\leq i\leq n}(t_{i}-t_{i-1})$. To build local quadratics, each element $I_j=[t_j,t_{j+1}]$ is associated with two end nodes and one midpoint
  node, whose local labels are
 \[  x_{j,1}=t_j, \quad x_{j,2}=t_{j+1/2}=:\frac{t_{j}+t_{j+1}}{2},\quad x_{j,3}=x_{j+1,1} = t_{j+1}.\]
 Note that in this ordering we have
\[
x_{j-1,3}=x_{j,1}, \quad x_{j,3}=x_{j+1,1}.
\]
On a non-interface element $I_j,j\neq k$ we let $\phi_{j,i}$  denote the standard local quadratic shape functions associated with $x_{j,i}$, $i=1,2,3$
such that $\phi_{j,i}(x_{j,l})=\delta_{i,l}$, i.e.,
%
%
%
\begin{eqnarray*}
\phi_{j,1} &=& \frac{x^{2} - x(x_{j,2} + x_{j,3})+x_{j,2}x_{j,3}}{(x_{j,2}-x_{j,1})(x_{j,3}-x_{j,1})},\\
\phi_{j,2} &=& \frac{x^{2} - x(x_{j,1} + x_{j,3})+x_{j,1}x_{j,3}}{(x_{j,1}-x_{j,2})(x_{j,3}-x_{j,2})},\\
\phi_{j,3} &=& \frac{x^{2} - x(x_{j,1} + x_{j,2})+x_{j,1}x_{j,2}}{(x_{j,1}-x_{j,3})(x_{j,2}-x_{j,3})}.
\end{eqnarray*}
For the interface element $I_k$ then the basis function $\phi_{k,i}$ associated with $x_{k,i}, i=1,2,3$ is defined so that
it is quadratic on $(x_{k,1},\alpha)$ and $(\alpha,x_{k,3})$ individually with
\[
[\phi_{k,i}]_{\alpha}= [\beta \phi'_{k,i}]_{\alpha}= [\beta \phi''_{k,i}]_{\alpha}=0.
\]
More specifically, for $j=k$ define
\begin{equation}\label{D}
 D=(\al-\xj)(\xjh-\xj)+(\al-\xj)(\xjo-\al)+\rho(\al-\xjh)(\al-\xjo)>0.
 \end{equation}


Then
\begin{equation*}
\phi_{j,1}(x)=
\begin{cases}
 \frac{(x-(\xjh+\xjo-\xj))(x-\xj)}{D}+1
   & x \in(\xj, \alpha),
\\
\frac{\rho(x-\xjh)(x-\xjo)}{D}
& x\in (\alpha,\xjo).
\end{cases}
\end{equation*}


\begin{equation*}
\phi_{j,2}(x)=
\begin{cases}
 \frac{(\xj-\al)(x-\al)(x-\xj) -\rho(\xjo-\al)(x-\xj)(x-(\al+\xjo-\xj))}{D\rho(\xjo-\xjh)}
   & x \in(\xj, \alpha),
\\
\frac{(\xj-\al)(x-\xjo)(x+(\xjo-\al-\xj))-\rho(\xjo-\al)(x-\al)(x-\xjo)}{D(\xjo-\xjh)}
& x\in (\alpha,\xjo).
\end{cases}
\end{equation*}


\begin{equation*}
\phi_{j,3}(x)=
\begin{cases}
 \frac{(\al-\xj)(x-\al)(x-\xj)-\rho(\al-\xjh)(x-\xj)(x-(\al+\xjh-\xj))}{D\rho(\xjo-\xjh)}
   & x \in(\xj, \alpha),
\\
\frac{(\al-\xj)(x-\xjh)(x-(\al+\xj-\xjh))-\rho(\al-\xjh)(x-\al)(x-\xjh)}{D(\xjo-\xjh)}
& x\in (\alpha,\xjo).
\end{cases}
\end{equation*}

Defining for $0\leq j\leq  n-1$ the local approximation space
\[      S_h(I_j)=span\{\phi_{j,i}, i=1,2,3\},\]
we see that it is the standard quadratics for non-interface elements and  is a piecewise quadratic space with a notable second derivative jump condition
 $[\beta \phi''_{k,i}]_{\alpha}=0$ for the interface element. This space was introduced in \cite{LinSun} and the extra condition is to guarantee optimal approximation error (Lin {\it et al.}\cite{LinSun} also defined other spaces, but they do not have optimal approximability). For each end node $t_j$ we define the global basis function $\phi_j$ to be one at the node and zero at other nodes so that
 $\phi_j\in S_h(I_j)$ and similarly for midpoint nodes. In this way we have constructed the global finite element space
  \[V_{h}=span \{\phi_{i},\phi_{i+1/2}\}_{i=0}^{n-1}\cap H^1_0(a,b)\]
   as an IFE space for approximating $p$. Consider the following  immersed  interface method for problem (\ref{eqn:First}):
Find $p_{h} \in V_{h} \subset H^{1}_{0}(a,b)$ such that
\begin{equation}\label{eqn:weak}
\int_{a}^{b}\beta p'_{h}v_{h}'dx+\int_a^b qp_hv_hdx = \int_{a}^{b}fv_{h}dx\qquad \forall v_{h}\in V_{h}.
\end{equation}
For simplicity, we assume that the coefficient
$\beta$ is positive and piecewise constant, i.e.,,
$$ \beta(x) = \beta^- ~~ \text{for } x\in (a,\alpha);~~~ \beta(x) = \beta^+ ~~
\text{for } x\in (\alpha,b).$$

Using the optimal approximation property of the interpolant $p_I\in V_h$ of $p$ in \cite{LinSun}, it is routine to prove the following theorem.

\begin{theorem} Assume that the solution to $p$ of (\ref{eqn:First})-(\ref{jump2}) further satisfies $[\beta p'']=0$ then
\[
\|p-p_{h}\|_{0,I} + h|p-p_{h}|_{1,I}\leq Ch^{3}\|p\|_{3,I},
\]
where the norm $\|p\|_{3,I}$ is piecewise defined as
\[      \|p\|^2_{s,I}:=\|p\|^2_{H^s(a,\alpha)}+\|p\|^2_{H^s(\alpha,b)},\quad  s=3.\]
\end{theorem}


\section{Construction of Approximate Flux}
In this section we construct the approximate flux $u_h$ of the exact flux $u=-\beta p'$.  In other word, we shall do flux recovery after we have obtained
the approximate pressure $p_h$. We first derive simple formulas for $u_h$ at the nodes of the elements and at the interface point. To build the global
$u_h$ we use piecewise quadratic interpolation. It is proper to point out at this stage that $u_h$ below is not defined as $-\beta p'_h$.

 To shorten presentation of the equations below we will collect the two terms in (\ref{eqn:First}) as
 \begin{equation}\label{collect}
       F:=f(x)-qp(x),\quad \mbox{ and its discrete version: } F_h:=f(x)-qp_h(x).
 \end{equation}
 Let us now multiply  (\ref{eqn:First}) by $\phi_{j,3}$ and integrate by parts over $I_{j} = [x_{j,1},x_{j,3}]$, $j\neq k$ to get
\begin{equation}\label{wkJ}
u(x^{-}_{j+1,1}) = u(x_{j,3}^{-}) = -\beta(x_{j,3}) p'(x_{j,3}) = -\int_{x_{j,1}}^{x_{j,3}}\beta p'\phi'_{j,3}dx + \int_{x_{j,1}}^{x_{j,3}}F\phi_{j,3}dx.
\end{equation}
Next we multiply  (\ref{eqn:First}) by $\phi_{j+1,1}$ and integrate by parts over $I_{j+1} = [x_{j+1,1},x_{j+1,2}]$, $j\neq k$ to get
\begin{eqnarray}\label{wkJp1}
u(x_{j+1,1}^{+}) &=& u(x^{+}_{j,3})= -\beta(x_{j+1,1}) p'(x_{j+1,1})\\
 &=& \int_{x_{j+1,1}}^{x_{j+1,3}}\beta p'\phi'_{j+1,1}dx - \int_{x_{j+1,1}}^{x_{j+1,3}}F\phi_{j+1,1}dx.\notag
\end{eqnarray}
Thus, if $p_{h}$ is a good approximate of $p$, we can define $u_{h}(x^{-}_{j+1,1})$ and $u_{h}(x^{+}_{j+1,1})$
on $I_{j}$ and $I_{j+1}$ respectively as,
\begin{eqnarray}
u_{h}(x^{-}_{j+1,1}) &=& u_{h}(x_{j,3}^{-}) \notag \\
&=& -\int_{x_{j,1}}^{x_{j,3}}\beta p_{h}'\phi'_{j,3}dx + \int_{x_{j,1}}^{x_{j,3}}F_h\phi_{j,3}dx, \label{uhj1m} \\
u_{h}(x_{j+1,1}^{+}) &=&  u_{h}(x^{+}_{j,3}) \notag \\
&=& \int_{x_{j+1,1}}^{x_{j+1,3}}\beta p_{h}'\phi'_{j+1,1}dx - \int_{x_{j+1,1}}^{x_{j+1,3}}F_h\phi_{j+1,1}dx.  \label{uhj1p}
\end{eqnarray}
Substituting $v_{h} = \phi_{j+1,1}$ into (\ref{eqn:weak}), we see that
\begin{eqnarray*}
\int_{x_{j,1}}^{x_{j+1,3}}\beta p_{h}'\phi'_{j+1,1}dx &=& \int_{x_{j,1}}^{x_{j+1,3}}F_h\phi_{j+1,1}dx, \\
-\int_{x_{j,1}}^{x_{j,3}}\beta p_{h}'\phi'_{j,3}dx + \int_{x_{j,1}}^{x_{j,3}}F_h\phi_{j,3}dx &=&
 \int_{x_{j+1,1}}^{x_{j+1,3}}\beta p_{h}'\phi'_{j+1,1}dx - \int_{x_{j+1,1}}^{x_{j+1,3}}F_h\phi_{j+1,1}dx.
\end{eqnarray*}
Thus, from (\ref{uhj1m}) and (\ref{uhj1p}) we can uniquely define
 $u_{h }(x_{j+1,1})=u_{h}(x^{-}_{j+1,1}) = u_{h }(x^{+}_{j+1,1})$.
 In a similar fashion by setting $v_{h} = \phi_{j,1}$ in (\ref{eqn:weak}), we can show that
 $u_h(x^{-}_{j,1}) = u_h(x^{+}_{j,1}) = u_h(x_{j,1})$.
%
%
%
%

For { the midpoint nodes}  with the basis function  by $\phi_{j,2}$,  over the interval
$[x_{j,1},x_{j,2}]$, $j\neq k$ we can define
\begin{equation*}
u_{h}(x^{-}_{j,2})=  -\int_{x_{j,1}}^{x_{j,2}}\beta p_{h}' \phi'_{j,2}dx +
\int_{x_{j,1}}^{x_{j,2}}F_h\phi_{j,2}dx,
\end{equation*}
and for the interval $[x_{j,2},x_{j,3}]$
\begin{equation*}
u_{h}(x_{j,2}^{+}) =  \int_{x_{j,2}}^{x_{j,3}}\beta p_{h}' \phi'_{j,2}dx -
\int_{x_{j,2}}^{x_{j,3}}F_h\phi_{j,2}dx.
\end{equation*}
Again setting $v_{h} = \phi_{j,2}$ in (\ref{eqn:weak}) gives
\begin{eqnarray*}
\int_{x_{j,1}}^{x_{j,3}}\beta p_{h}'\phi'_{j,2}dx &=& \int_{x_{j,1}}^{x_{j,3}}F_h\phi_{j,2}dx \\
-\int_{x_{j,1}}^{x_{j,2}}\beta p_{h}'\phi'_{j,2}dx + \int_{x_{j,1}}^{x_{j,2}}F_h\phi_{j,2}dx &=&
 \int_{x_{j,2}}^{x_{j,3}}\beta p_{h}'\phi'_{j,2}dx - \int_{x_{j,2}}^{x_{j,3}}F_h\phi_{j,2}dx.
\end{eqnarray*}
Thus $u_{h}(x_{j,2}):=u_h(x^{-}_{j,2}) = u_h(x^{+}_{j,2})$.
So we have $u_{h}(x_{j,i}) = u^{-}_{h}(x_{j,i}) = u^{+}_{h}(x_{j,i})$ for $i=1,2,3$.

On a non-interface element, $u_h$ at a non-nodal point is defined  by quadratic interpolation:
$I_j$ $j\neq k$, we define $u_{h}(x)$ as
\begin{equation} \label{NnIntFlux}
u_{h}(x) := u_{h}(x_{j,1})\phi_{j,1}(x) + u_{h}(x_{j,2})\phi_{j,2}(x) + u_{h}(x_{j,3})\phi_{j,3}(x) \quad \textrm{for $x\in I_j.$}
\end{equation}
%
%
%
To approximate flux on the interface element, first we multiply (\ref{eqn:First}) by $\phi_{k,1}$ and integrate by parts over $I_{k} = [x_{k,1},x_{k,3}]$
to get
\begin{eqnarray*}\label{wkJp11}
u(x_{k-1,3}^{+}) &=&u(x^{+}_{k,1})\\
&=& -\beta(x_{k,1}) p'(x_{k,1}) \\
&=& \int_{x_{k,1}}^{\alpha}\beta p'\phi'_{k,1}dx
+ \int^{x_{k,3}}_{\alpha}\beta p'\phi'_{k,1}dx
-  \int_{x_{k,1}}^{\alpha}F\phi_{k,1}dx - \int_{\alpha}^{x_{k,3}}F\phi_{k,1},
\end{eqnarray*}
where we have used the jump condition(\ref{jump1}).
For the adjacent non-interface element  $I_{k-1} = [x_{k-1,1},x_{k-1,3}]$, we get
\begin{eqnarray*}\label{wkJ1}
u(x^{-}_{k-1,3})=u(x_{k,1}^{-}) &=& -\beta(x_{k,1}) p'(x_{k,1}) \\
&=& -\int_{x_{k-1,1}}^{x_{k-1,3}}\beta p'\phi'_{k-1,3}dx +
\int_{x_{k-1,1}}^{x_{k-1,3}}F\phi_{k-1,3}dx.
\end{eqnarray*}
Then if $p_{h}$ is a good approximate for $p$ it is natural to define
\begin{eqnarray}
  u_{h}(x_{k-1,3}^{+}) &=& u_{h}(x^{+}_{k,1})\notag \\
 &=&
\int_{x_{k,1}}^{\alpha}\beta p_{h}'\phi'_{k,1}dx + \int_{\alpha}^{x_{k,3}}\beta p_{h}'\phi'_{k,1}dx
-\int_{x_{k,1}}^{\alpha}F_h\phi_{k,1}dx - \int_{\alpha}^{x_{k,3}}F_h\phi_{k,1}dx,
 \label{uhkp1}   \\
u_{h}(x^{-}_{k-1,3}) &=& u_{h}(x_{k,1}^{-}) \notag\\
& =& -\int_{x_{k-1,1}}^{x_{k-1,3}}\beta p_{h}'\phi'_{k-1,3}dx +
\int_{x_{k-1,1}}^{x_{k-1,3}}F_h\phi_{k-1,3}dx. \label{uhkm1}
\end{eqnarray}
Again setting $v_{h} = \phi_{k,1}$ in (\ref{eqn:weak}) as in the interface element gives
\begin{align*}
-\int_{x_{k-1,1}}^{x_{k-1,3}}&\beta p_{h}'\phi'_{k-1,3}dx +
\int_{x_{k-1,1}}^{x_{k-1,3}}F_h\phi_{k-1,3}dx.  \\
&=
\int_{x_{k,1}}^{\alpha}\beta p_{h}'\phi'_{k,1}dx + \int_{\alpha}^{x_{k,3}}\beta p_{h}'\phi'_{k,1}dx
-\int_{x_{k,1}}^{\alpha}F_h\phi_{k,1}dx - \int_{\alpha}^{x_{k,3}}F_h\phi_{k,1}dx.
\end{align*}
Thus, from (\ref{uhkp1}) and (\ref{uhkm1}) $u_{h}(x^{-}_{k,1}) = u_{h}(x^{+}_{k,1}) = u_{h}(x_{k,1})$. Similarly, it is easy to see that
$u_h(x^{-}_{k,2}) = u_h(x^{+}_{k,2}) = u_h(x_{k,2})$ and
$u_h(x^{-}_{k,3}) = u_h(x^{+}_{k,3}) = u_h(x_{k,3})$.

Since we want higher precision for $u_h(\alpha)$, it will not be defined by interpolation. According to the experience \cite{Chou}, we define
\begin{equation}\label{u_h_alp}
u_h(\alpha)=u_h(x_{k,1})+\int_{x_{k,1}}^\alpha F_h(x)dx,
\end{equation}
which is based on
\begin{equation}\label{u_alp}
u(\alpha)=u(x_{k,1})+\int_{x_{k,1}}^\alpha F(x)dx.
\end{equation}
Then, we can define the flux approximation over the interface element $I_{k}=[x_{k},x_{k+1}]$ by a cubic interpolating polynomial
at $\alpha$ and the three nodal points. Of course there are other choices based on how smooth $u$ can be. For example, one can first define by interpolation
\begin{equation} \label{IntFlux}
u_{h}(x) = u_{h}(x_{k,1})\phi_{k,1}(x) + u_{h}(x_{k,2})\phi_{k,2}(x) + u_{h}(x_{k,3})\phi_{k,3}(x) \quad \textrm{for $x\in I_k$}
\end{equation}
and get $u_h(\alpha)$ by evaluation. This approach is more natural for higher dimensional case.
In any case it is not hard to see that
 the resulting $L_2$ norm error estimates can be derived once the pointwise error estimates at nodal and/or interface points are known, which will be addressed
 in the next section.

%
%
%

%
For completeness, let us include in the next theorem a possible second order $L^2$ estimate without the knowledge of pointwise errors.
Thus it will be justified to call a point $x$ superconvergent one if
\[  u(x)-u_h(x)=O(h^{2+\sigma}), \mbox{  for some   } \sigma >0.\]

\begin{theorem}\label{thm4.2}
Let $u$ be the exact flux and $u_{h}$ be the approximated flux as defined by (\ref{NnIntFlux}) and (\ref{IntFlux}) for interface and non-interface elements respectively. Then
\[
\|u-u_{h}\|_{0,I} \leq Ch^{2}\|p\|_{3,I}.
\]
\end{theorem}
\begin{proof}
%
%
%
We give a proof for the interface element $I_{k}=[x_{k,1},x_{k,3}]$. The non-interface follows similarly. On the interval $(x_{k,1},x_{k,3})$
\begin{eqnarray}\label{FluxErrIntfc}
|(u_{h}-u)(x)| &\leq& \sum_{i=1}^{3}|(u_{h}- u)(x_{k,i})\phi_{k,i}| + \sum_{i=1}^{3}|(u(x_{k,i}) - u(x))\phi_{k,i}|
\end{eqnarray}
where we have used $\sum_{i=1}^{3}\phi_{k,i}=1$.
Let $\beta^*=\max\{\beta^-,\beta^+\}$.
For $i=1,3$
\begin{eqnarray*}
\|(u(x_{k,i}) - u_{h}(x_{k,i}))\phi_{k,i}\|_{0,I_{k}} &=& \|(\int_{x_{k,1}}^{x_{k,3}}\beta (p'_{h}- p')\phi'_{k,i}dx)\phi_{k,i}\|_{0,I_{k}}+ \\
&&\quad + \|(\int_{x_{k,1}}^{x_{k,3}}q (p-p_h)\phi_{k,i}dx)\phi_{k,i}\|_{0,I_{k}}\\
&=&J_1+J_2,
\end{eqnarray*}
where
\begin{eqnarray*}
J_1 &\leq& \beta^{*}|\int_{x_{k,1}}^{x_{k,3}}(p'_{h}-p')\phi'_{k,i}dx|\|\phi_{k,i}\|_{0,I_{k}} \\
&\leq&  \beta^{*}\|p'_{h}-p'\|_{0,I_{k}}\|\phi'_{k,i}\|_{0,I_{k}}\|\phi^{}_{k,i}\|_{0,I_{k}}.
\end{eqnarray*}

Now note it is not hard to see from (\ref{D}) that $D\ge Ch^2$ with the constant $C$ independent of $h$ and $\alpha$ and so
\begin{eqnarray*}
\|\phi^{}_{k,i}\|_{0,I_{k}} &\leq& Ch_{k}^{1/2} \leq Ch^{1/2} \\
\|\phi'^{}_{k,i}\|_{0,I_{k}} &\leq& Ch_{k}^{-1/2} \leq Ch^{-1/2}.
\end{eqnarray*}
With this in mind, we have
\[   J_1\leq Ch^{2}\|p\|_{3,I_k},\qquad  J_2\leq Ch^{4}\|p\|_{3,I_k},\]
where
\[  \|p\|^2_{3,I_k}:=\|p\|^2_{H^3(x_{k,1},\alpha)}+\|p\|^2_{H^3(\alpha,x_{k,3})}.\]
Hence
\begin{eqnarray}
\|u(x_{k,i}) - u_{h}(x_{k,i})\phi^{}_{j,i}\|_{0,I_{k}} &\leq& C\beta^{*}\|p'_{h} - p'\|_{0,I_{k}}  \notag \\
&\leq& Ch^{2}\|p\|_{3,I_{k}}. \label{InErr01}
\end{eqnarray}
Similar estimate holds when $i=2$.

As for the second term on the right of (\ref{FluxErrIntfc}), we have by the mean value theorem with some $\xi\in (x_{k,1},x_{k,2})$ so that
\[
|(u(x_{k,i}) - u(x))\phi^{}_{k,i}| = |x_{k,i}-x||u'(\xi)||\phi^{}_{k,i}|.
\]
Then using the one dimensional Sobolev imbedding theorem to extract a factor of $h^{1/2}$ we have
\begin{equation}\label{InErr02}
\|(u(x_{k,i})-u(x))\phi^{}_{k,i}\| \leq h^{3/2}\|\phi^{}_{k,i}\|\|p\|_{3,I_{k}} \leq h^{2}\|p\|_{3,I_{k}}.
\end{equation}
Combining this with (\ref{InErr01}) and (\ref{InErr02}) on (\ref{FluxErrIntfc}) gives
\[
\|u-u_{h}\|_{0,I_k} \leq Ch^{2}\|p\|_{3,I_k}.
\]
\end{proof}

\subsection{Pointwise  Errors at Nodes and Interface Point}
In this section we estimate approximate pressure and flux  errors at nodes and interface point.
Superconvergence points of pressure and flux are shown to be end nodes.
\begin{theorem}\label{thm4.3}
Consider problem (\ref{eqn:First}) with $q=0$.
 Let the approximate flux $u_{h}$ be defined by (\ref{NnIntFlux}) and (\ref{uhkp1})-(\ref{u_h_alp}). Let $p_h$ be the approximate pressure defined by (\ref{eqn:weak})
 and  $p$ be the pressure defined by (\ref{weak}).
 Suppose that
  the coefficient $\beta$ is piecewise constant. Then, the following statements hold.\\

{\bf (i) Exactness of approximate pressure $p_h$ at the end nodes:}
  \[       p(t_i)=p_h(t_i)\quad  \quad \mbox{ at all end nodes} \quad t_i, i=0,\ldots,n.\]

  {\bf (ii)
  Error in the approximate pressure at the interface point:}
  \[   |p(\alpha)-p_h(\alpha)|\leq Ch^{2.5},\]
  and at the midpoints
  \[  |p(x)-p_h(x)|\leq Ch^{2.5},\quad  x=t_{i+1/2},\, i=0,\ldots,n-1.\]

 { \bf (iii) Uniform error at the end nodes and interface point:}
The errors at
the end nodes and interface point are identical, i.e.
\begin{equation}\label{cnst}
E(x)=u(x)-u_{h}(x) =C
\end{equation}
for all  $x=t_i,i=1\ldots n,\alpha$.

   {\bf (iv) Exactness of approximate flux at the nodes and interface point.} \\
   The constant $C$ in (\ref{cnst}) is zero, i.e.,
   \[  u(x)=u_h(x)\qquad
  \mbox{ for  all }  x=t_i, i=0,\ldots,n\quad \mbox{ and } \alpha.\]
\end{theorem}

\begin{proof}
Fix $\xi\in(a,b)$ and let $G(x,\xi)$ be the Green's function satisfying
\[ a(G(\cdot,\xi),v)=<\delta(x-\xi),v>,\quad v\in H^1_0(a,b).\]
By working out the closed form of $G$ satisfying the classical formulation
\begin{equation}\label{Gjump}
  -(\beta G')'=\delta(x-\xi),\quad [G]_\alpha=0,  \quad[\beta G']_\alpha=0, \quad G(a,\xi)=G(b,\xi)=0,
  \end{equation}
we see that $G$ can be expressed in terms of $\int_d^x\frac{1}{\beta(t)}dt$ for different $d$. For instance, the Green's
function for $(a,b)=(0,1)$ and $\xi<\alpha$ takes the form \cite{Chou}
\begin{equation}\label{green}
G(x,\xi)= \left\{ \begin{aligned}
 A\int_0^x\frac{1}{\beta(t)}dt,&\qquad 0<x\leq \xi,\\
(A-1)\int_{\xi}^x\frac{1}{\beta(t)}dt+A\int_0^\xi\frac{1}{\beta(t)}dt ,&\qquad \xi\leq x\leq \alpha,\\
(1-A)\int_x^1\frac{1}{\beta(t)}dt,&\qquad \alpha\leq x \leq 1,\\
\end{aligned} \right.
\end{equation}
where
\[A=\frac{\int_\xi^1\frac{1}{\beta(t)}dt}{\int_0^1\frac{1}{\beta(t)}dt}.\]
Note that $G(x,t_i)$ is piecewise linear when $\beta$ is piecewise constant.
%
Now let $G=G(x,t_i)$ and use Galerkin orthogonality property, then
\[   e(t_i)=(\delta(x-t_i), e)=a(G, e)=0,\]
since $G\in V_h$ (when $\beta(t)$ is piecewise constant, $G$ is piecewise linear and satisfies all the jump conditions including $[\beta G'']=0$).
This proves (i).

As for (ii), without loss of generality let's assume $\alpha$ lies in $(t_k,t_{k+1/2})$. At the interface point $\alpha$, $G(x,\alpha)$ is no longer in $V_h$ since $[\beta G']_\alpha=-1$, not zero (Equations (\ref{Gjump}) and
(\ref{green}) have obvious modifications).
In this case, let $G_h:=G-b_h$, where $b_h$ is the bubble function with support $[t_k,t_{k+1}]$, piecewise linear, and $[\beta b_h']_\alpha=-1$, i.e.,
\[  b_h(x)=\left\{
\begin{aligned}{\mathcal A}(t_{k+1}-\alpha)(x-t_k),&\quad x\in [t_k,\alpha]\\
{\mathcal A} (\alpha-t_k)(t_{k+1}-x),&\quad x\in[\alpha,t_{k+1}]\\
0&, \quad \mbox{ otherwise},
\end{aligned}
\right.
\]
where
\[  {\mathcal A}=\frac{1}{\beta^-(t_{k+1}-\alpha)+\beta^+(\alpha-t_k)}.\]
Noting that now $G_h\in V_h$ and $|b_h|_{1,I}\leq Ch^{1/2}$, we have
\[   e(\alpha)=(\delta(x-\alpha), e)=a(G, e)=a(G-G_h,e)=a(b_h,e)\leq C |b_h|_{1,I}|e|_{1,I}\leq Ch^{5/2}\|p\|_{3,I}.\]

For a midpoint $\xi=t_{i+1/2}$, its associated Green's function $G(x,\xi)$ is neither in $H^2(t_i,t_{i+1})$ nor in $V_h$, being piecewise linear in $(t_i,t_{i+1})$. To approximate $G$, we construct $G_h\in V_h$ such that $G=G_h$ over $I-(t_i,t_{i+1})$, and on $(t_i,t_{i+1})$ $G_h$ is defined as the quadratic interpolant to $G$ at the nodes $t_i,t_{m},t_{i+1}$, $t_m=t_{i+1/2}$. Thus using the local ordering
\begin{equation}\label{interp}
   G_h(x)=\sum_{j=1}^3G(x_{i,j})\phi_{i,j}(x)\qquad \forall x\in [t_i,t_{i+1}].
   \end{equation}
In addition, it is easy to see that
\begin{equation*}
G_h(x)-G(x)=
\begin{cases}
\frac{s}{2}(x-t_i)(x-t_m),\qquad  x\in[t_i,t_m],\\
\frac{s}{2}(x-t_m)(x-t_{i+1}),\qquad x\in [t_m,t_{i+1}],
\end{cases}
\end{equation*}
where the second derivative $s=G_h''$ can be computed from (\ref{interp}) (or centered difference by inspection !!) and
\begin{eqnarray*}
s&=&\frac{4}{h^2}(G(t_i)-G(t_m)+G(t_{i+1})-G(t_m))\\
&=& \frac{2}{h}(G'(t_m^+)-G'(t_m^-))\\
&=& \frac{2}{h\beta}(\beta G'(t_m^+)-\beta G'(t_m^-))\\
&=&\frac{-2}{h\beta}.
\end{eqnarray*}
Consequently,
\begin{equation}\label{bdd}
|G(x)-G_h(x)|\leq Ch \qquad \mbox{    and   } \quad  |G'_h(x)-G'(x)|\leq C.
\end{equation}
With this in mind we see that
\begin{equation}\label{H1}
   |G-G_h|_{1,I}=|G-G_h|_{1,I_i}\leq C h^{0.5}\qquad I_i=(t_i,t_{i+1}).
\end{equation}

Hence
\[   e(\xi)=(\delta(x-\xi), e)=a(G,e)=a(G-G_h,e)\leq Ch^{2.5}\|p\|_{3,I}.\]
 This completes the proof of (ii).

Next, we prove (iii).
Let $E(x)=u(x)-u_{h}(x)$. By (\ref{wkJ})-(\ref{uhj1p}),
\[
E(x_{j,3}) = -\int_{x_{j,1}}^{x_{j,3}}\beta(p'-p'_{h})\phi'_{j,3}dx+\int_{x_{j,1}}^{x_{j,3}}q(p_h-p)\phi_{j,3}dx
\]
and
\begin{eqnarray*}
E(x_{j+1,3}) &=& -\int_{x_{j+1,1}}^{x_{j+1,3}}\beta(p'-p'_{h})\phi'_{j+1,3}dx +\int_{x_{j+1,1}}^{x_{j+1,3}}q(p_h-p)\phi_{j+1,3}dx\\
&=& \int_{x_{j+1,1}}^{x_{j+1,3}}\beta(p'-p'_{h})(\phi'_{j+1,1} +\phi'_{j+1,2} )dx\\
&&+\int_{x_{j+1,1}}^{x_{j+1,3}}q(p_h-p)(1-(\phi_{j+1,1} +\phi_{j+1,2})) dx.
\end{eqnarray*}
Assembling contributions from the local shape functions, we have in terms of global shape functions $\phi_{j+\frac 3 2},\phi_{j+1}$
\begin{align}\notag
E&(x_{j+1,3})-E(x_{j,3}) \\\notag
&= a(p-p_{h},\phi_{j+\frac 3 2}) + a(p-p_{h},\phi_{j+1})+\int_{x_{j+1,1}}^{x_{j+1,3}}q(p_h-p) dx \\
&=\int_{x_{j+1,1}}^{x_{j+1,3}}q(p_h-p) dx.\label{unif3}
\end{align}
Hence $E(x_{j+1,3})=E(x_{j,3})$ when $q=0$.
The above argument holds for both interface and non-interface elements.
Finally, subtracting (\ref{u_h_alp}) from (\ref{u_alp}) we have
  \[
u(\alpha)-u_h(\alpha)=E(t_k)=C.\]
This completes the proof of (iii).

Now we prove (iv). Due to (iii), it suffices to look at
\begin{eqnarray*}
   E(t_0)&=&u(a)-u_h(a)\\
   &=&\int_a^{t_1}\beta(p'-p_h')\phi_0'dx\quad \mbox{  $\phi_0$ is the nodal quadratic shape function at $t_0$}\\
   &=&\frac{2\beta^-}{h^2}\int_a^{t_1}(p'-p_h')(2x-t_{1/2}-t_1)dx\\
   &=&\frac{2\beta^-}{h^2}\int_a^{t_1}(p'-p_h')(2x)dx   \quad \mbox{ by (i)}\\
   &=& 0,
   \end{eqnarray*}
   where the last equality is derived as follows.
   Since $a(p-p_h,\phi_{1/2})=0$,
   we have
   \begin{eqnarray*}
 0&=& \frac{-4\beta^-}{h^2}\int_a^{t_1}(p'-p_h')(2x-a-t_1)dx\\
 &=&\frac{-4\beta^-}{h^2}\int_a^{t_1}(p'-p_h')(2x)dx \quad \mbox{ by (i)}.
   \end{eqnarray*}
   This completes the proof of (iv) for end nodes.

  Subtracting (\ref{u_h_alp}) from (\ref{u_alp}), we have
  \[
u(\alpha)-u_h(\alpha)=E(t_k)=0.\]

%
  \end{proof}

We now move to the general case.
\begin{theorem}\label{thm4.4}
Consider problem (\ref{eqn:First}) with $q\geq 0$.
 Let the approximate flux $u_{h}$ be defined by (\ref{NnIntFlux}) and (\ref{uhkp1})-(\ref{u_h_alp}), $p_h$ defined by (\ref{eqn:weak}), and $p$ defined by (\ref{weak}).
 Suppose that
  the coefficient $\beta$ is piecewise constant. Then the following statements hold.\\

{\bf (i) Fourth order convergence rate for approximate pressure $p_h$ at the end nodes.}
  \[       |p(t_i)-p_h(t_i)|\leq Ch^4, \quad  \quad \mbox{ at all end nodes} \quad t_i, i=0,\ldots,n.\]

  {\bf (ii)
  The error in the approximate pressure at the interface point satisfies}
  \[   |p(\alpha)-p_h(\alpha)|\leq Ch^{2.5},\]
  and at the midpoints
  \[  |p(x)-p_h(x)|\leq Ch^{2},\quad  x=t_{i+1/2},\, i=0,\ldots,n-1.\]

 { \bf (iii) Almost uniform error at the end nodes and interface point.}
Define $E(x):=u(x)-u_h(x)$. Then
\begin{equation}\label{cnst1}
E(t_i)=E(t_{i-1}) +O(h^{3.5})\qquad i=1,\ldots n,
\end{equation}
 and
\begin{equation}  E(\alpha)=E(t_k)+O(h^{3.5}).\end{equation}

   {\bf (iv) Third order superconvergence rate of approximate flux at the nodes and interface point.} \\
   \[  |u(x)-u_h(x)|\leq Ch^3 \qquad
  \mbox{ for  all }  x=t_i, i=0,\ldots,n\quad \mbox{ and } \alpha.\]
\end{theorem}
\begin{proof}
Let $G(x,\xi)$ be the Green's function satisfying
\[ a(G,v)=<\delta(x-\xi),v>,\quad v\in H^1_0(a,b).\]
By working out the closed form of $G$ satisfying the classical formulation
\begin{equation}\label{Gjump1}
  -(\beta G')'+qG=\delta(x-\xi),\quad [G]_\alpha=0,  \quad[\beta G']_\alpha=0, \quad G(a,\xi)=G(b,\xi)=0,
  \end{equation}
just as in (\ref{Gjump}), it is not hard to see that $G$ is a linear combination of smooth functions in $(a,\xi),(\xi,\alpha),(\alpha,b)$
and $G(x,t_i)\in H^3(\Omega)$, for $\Omega=(t_j,t_{j+1}),j\not=k$ and $\Omega=(t_k,\alpha), (\alpha,t_{k+1})$. Similar conclusions hold when $\xi>\alpha$.
Also observe that $[\beta G'']_\alpha=0$ as well. In fact, the classical interpretation of the Green's function implies that
\[   -\beta G''+qG=0  \mbox{ on}\quad (a,\alpha),(\alpha,\xi),\]
since $\beta$ is piecewise constant. So $[\beta G'']_\alpha=[qG]_\alpha=0$.
Now let $G=G(x,t_i)$ and without loss
of generality let's assume $\alpha$ lies in $(t_k,t_{k+1/2})$.
By the local approximation estimates, there exists $ G_h\in V_h$, the interpolant of $G$, such that

\begin{equation}\label{gr} \|G-G_h\|_{1,\Omega}\le C h^2 ||G||_{3,\Omega}
\end{equation}
for all the $\Omega$'s listed above.
Hence
\[   e(t_i)=(\delta(x-t_i), e)=a(G, e)=a(G-G_h,e)\leq Ch^{4}.\]
This completes the proof of (i).
The proof of the interface case in (ii) is similar to that of Thm \ref{thm4.3}.
However, for the midpoint case, the estimate in (\ref{bdd}) or (\ref{gr}) is not applicable since now the Green's function is neither piecewise linear
nor in $H^3(t_i,t_{i+1})$.
Instead we have
\[   e(t_{i+1/2})=(\delta(x-t_{i+1/2}), e)=a(G,e)\leq Ch^{2}\|p\|_{3,I}.\]
Table 3 in the next section will confirm this order is the best we can achieve.
 The statement (iii) is a direct consequence of (\ref{unif3}).

 We now prove (iv).

\begin{eqnarray*}
   E(t_0)&=&u(a)-u_h(a)\\
   &=&\int_a^{t_1}\beta(p'-p_h')\phi_0'dx+\int_a^{t_1}q(p-p_h)\phi_0dx.\\
   &=&J_1+J_2.
      \end{eqnarray*}
      As in (iv) of Thm \ref{thm4.3} we need a refined estimate for the first term $J_1$ on the right side.
      First observe that $a(p-p_h,\phi_{1/2})=0$ with $\phi'_{1/2}=-\frac{4}{h^2}(2x-a-t_1)$ leads to a relation
      \begin{eqnarray}\label{2x}
        \int_a^{t_1}\beta(p-p_h)'(2x)dx&=&\int_a^{t_1}\beta(p-p_h)'(a+t_1)dx
        +\frac{h^2}{4}\int_a^{t_1}q(p-p_h)\phi_{1/2}dx\notag\\
        &=&(a+t_1)\beta(p-p_h)(t_1)+\frac{h^2}{4}\int_a^{t_1}q(p-p_h)\phi_{1/2}dx.\label{24}
        \end{eqnarray}
      Further, with $\phi'_0=\frac{2}{h^2}(2x-t_1-t_{1/2})$ we have
      \begin{eqnarray*}
      J_1&=&\frac {2}{h^2} \int_a^{t_1}\beta(p-p_h)'(2x-t_1-t_{1/2})dx\\
      &=&\frac {2}{h^2} \int_a^{t_1}\beta(p-p_h)'(2x)dx-\frac {2}{h^2} \beta(t_1+t_{1/2})(p-p_h)(t_1)\\
      &=&\frac {2}{h^2}(a+t_1)\beta(p-p_h)(t_1)+\frac 12 \int_a^{t_1}q(p-p_h)\phi_{1/2}dx \\
      &&-\frac {2}{h^2} \beta(t_1+t_{1/2})(p-p_h)(t_1)  \quad \mbox{ by (\ref{24})}\\
      &=&\frac {2}{h^2}[\beta(a-t_{1/2})(p-p_h)(t_1)]+\frac 12 \int_a^{t_1}q(p-p_h)\phi_{1/2}dx.
      \end{eqnarray*}

Thus
\[  J_1\leq Ch^{-2}h h^4+Ch^3h^{0.5}\leq Ch^3,\]
and  $J_2\leq Ch^3$ imply
\[  E(t_0)\leq Ch^3.\]
The rest of the proof follows from (\ref{cnst1}) and iteration.
 This completes the proof of (iv).
\end{proof}

\section{Numerical examples}

{\bf Problem 1.} Consider
\[
-(\beta p')' = f(x)=x^m,  \quad p(0)=p(1)=0,
\]
where $m$ is a nonnegative integer. The interface point is located at $\alpha$ and

\begin{equation*}
\beta(x) =  \begin{cases}
      \beta^{-} & x\in [0,\alpha), \\
	 \beta^{+} & x\in (\alpha,1].
   \end{cases}
\end{equation*}

The exact solution is

\begin{equation}\label{exctP}
p(x) =  \begin{cases}
     \displaystyle\frac{-1}{(m+1)(m+2)\beta^{-}}x^{m+2}+\frac{t^{-}}{\beta^{-}}x & x\leq\alpha, \\
	\displaystyle \frac{-1}{(m+1)(m+2)\beta^{+}}x^{m+2}+\frac{t^{+}}{\beta^{+}}x - \frac{t^{+}}{\beta^{+}} -  \frac{-1}{(m+1)(m+2)\beta^{+}}  & x\geq\alpha,
   \end{cases}
\end{equation}
where
\begin{align*}
t^{+}&=t^{-}\\
&=\left(\frac{\alpha-1}{\beta^{+}}-\frac{\alpha}{\beta^{-}}\right)
\left(
 \frac{-\alpha^{m+2}}{(m+1)(m+2)\beta^{-}} + \frac{\alpha^{m+2}}{(m+1)(m+2)\beta^{+}} -
  \frac{1}{(m+1)(m+2)\beta^{+}}
\right).
\end{align*}
The flux
\begin{equation}\label{flxP}
u(x) = -\beta p'(x) =\frac{1}{m+1}x^{m+1} - t^{-}
\end{equation}
is smooth over $[0,1]$. For the numerical runs, we set $\beta^{-}=100$, $\beta^{+}=1$, $f(x)=x^{m}$, $\alpha=1/3$ and calculate the maximum  pressure and flux error at nodes
\begin{eqnarray*}
 pErrEndNodes &=&  \max_{1\leq i\leq n-1} |p(t_{i})-p_{h}(t_{i})|, \\
 pErrMidNodes &=& \max_{1\leq i\leq n-1} |p(t_{i+1/2})-p_{h}(t_{i+1/2})|,  \\
 uErrEndNodes &=& \max_{1\leq i\leq n-1} |u(t_{i})-u_{h}(t_{i})|,  \\
 \end{eqnarray*}
respectively. At the interface point $\alpha$ errors are given by
 \begin{eqnarray}
 pErr@alp &=& |p(\alpha)-p_{h}(\alpha)|, \\
 uErr@alp &=& |u(\alpha)-u_{h}(\alpha)|.
 \end{eqnarray}
In Tables \ref{fig:P1press} and \ref{fig:P1flx} below we list the error at the nodes and the interface points for different mesh sizes and $m$ values for pressure and flux, respectively.  The pressure at the end nodes and the flux both at the end nodes and at the interface point numerical values are exact, as predicted by Thm \ref{thm4.3}. However, for pressure at the midpoint nodes and interface point numerical results are better than the theoretic estimates.

\begin{table}[h]
\centering
    \begin{tabular}{|c||c|c|c|c|c|c|}
    \hline
    Problem 1 & h=1/16 & h=1/32 & h=1/64 & h=1/128 & m & order  \\
    \hline\hline
    pErrEndNodes &  1.0755e-13 & 2.9143e-14 & 1.3910e-14 & 3.2916e-15 & 2 & $\approx$ exact \\
    \hline
     pErrEndNodes & 4.4541e-13 & 6.5573e-14 & 1.5613e-14 & 4.1633e-15 & 5 & $\approx$ exact \\
    \hline
     pErrEndNodes & 1.9241e-13 & 2.5717e-14 & 6.099e-15 & 1.9949e-15 & 10 &  $\approx$ exact \\
     \hline
     pErrMidNodes & 1.5895e-08 & 9.9341e-10 & 6.2088e-11 & 3.8880e-12 & 2 & $\approx$ 4 \\
     \hline
     pErrMidNodes & 1.4455e-07 & 9.4764e-09 & 6.0645e-10 & 3.8352e-11 & 5 & $\approx$ 4 \\
     \hline
      pErrMidNodes & 5.5636e08 & 3.9438e-08 & 2.6245e-09 & 1.6925e-10 & 10 & $\approx$ 4 \\
      \hline
    pErr@alp   & 1.0282e-06 & 1.2412e-07 & 1.5790e-08 & 1.9565e-09   & 2 & $\approx$ 4 \\
    \hline
   pErr@alp   & 1.0260e-07 & 1.1108e-08 & 1.4884e-09 & 1.4884e-09   & 5 & $\approx$ 4 \\
   \hline
     pErr@alp   &  9.7359e-10  & 8.6795e-11 & 1.2635e-11 & 1.4572e-12 & 10 & $\approx$ 4 \\
     \hline
    \end{tabular}
    \caption{Maximum error at the nodes and the interface point of approximate pressure }
    \label{fig:P1press}
\end{table}

\begin{table}[h]
\centering
    \begin{tabular}{|c||c|c|c|c|c|c|}
    \hline
    Problem 1 & h=1/16 & h=1/32 & h=1/64 & h=1/128 & m & order  \\
    \hline\hline
    uErrEndNodes & 3.9077e-13 & 9.8865e-14 & 2.0622e-14 & 4.6352e-15 & 2 & $\approx$ exact \\
    \hline
    uErrEndNodes & 1.2890e-13 & 2.4626e-14 & 2.4626e-14 & 7.2650e-15 & 5 & $\approx$ exact \\
    \hline
    uErrEndNodes & 1.2244e-14 & 4.2251e-14 & 2.8484e-15 & 3.8858e-16 & 10 & $\approx$ exact \\
    \hline
 uErr@alp   & 1.0729e-13 & 2.4786e-14 & 1.4710e-15 & 7.9381e-15  & 2 & $\approx$ exact \\
 \hline
  uErr@alp   & 3.4445e-14 & 7.0742e-15 & 1.1657e-15 &  3.4348e-16   & 5 & $\approx$ exact \\
  \hline
   uErr@alp   &  1.3572e-14 & 2.6056e-15 & 4.8399e-16 &  3.0184e-16 & 10 & $\approx$ exact \\
   \hline
    \end{tabular}
    \caption{Maximum error at the nodes and the interface point $\alpha$ of approximate flux }
     \label{fig:P1flx}
\end{table}

{\bf Problem 2.} Consider
\[
-(\beta p')' + qp= f(x),  \quad p(0)=p(1)=0,
\]
where $m$ and $\alpha$ are defined in a same way as in Problem 1. We used the same exact solution $p(x)$ and $u(x)$ defined in (\ref{exctP}) and (\ref{flxP}).  For the numerical simulation we set $q=1$ and $f(x)=x^m+p(x)$ and $\beta$ and $\alpha$ values are same as in Problem 1. For Problem 2, convergence rates for pressure at nodes are as predicted in Thm \ref{thm4.4}, whereas for the flux numerical values have higher convergence rates at the end nodes and at the interface point.
\begin{table}[h]
\centering
    \begin{tabular}{|c||c|c|c|c|c|c|}
    \hline
    Problem 2 & h=1/16 & h=1/32 & h=1/64 & h=1/128 & m & order  \\
    \hline\hline
    pErrEndNodes &  1.5322e-08 & 9.7490e-10 & 6.1512e-11 &  3.8833e-12 & 2 & $\approx$ 4 \\
    \hline
     pErrEndNodes & 1.4261e-07 & 9.4103e-09 & 6.0430e-10 & 3.8283e-11 & 5 & $\approx$ 4 \\
    \hline
     pErrEndNodes & 5.5233e-07& 3.9290e-08 & 2.6194e-09 & 1.6908e-10& 10 &  $\approx$ 4 \\
     \hline
    pErrMidNodes & 1.5774e-04 & 4.0059e-05 & 1.0093e-05 & 2.5332e-06  & 2 & $\approx$ 2 \\
    \hline
   pErrMidNodes & 3.5886e-04 & 9.5547e-05 & 2.4648e-05 & 6.2592e-06  & 5 & $\approx$ 2 \\
   \hline
   pErrMidNodes & 6.1493e-04 & 1.7680e-04 & 4.7412e-05 & 1.2277e-05 & 10 & $\approx$ 2 \\
   \hline
    pErr@alp   & 1.0282e-06 & 1.2412e-07 & 1.5790e-08 & 1.9565e-09 & 2 & $\approx$ 4 \\
   \hline
     pErr@alp   & 1.0260e-07 & 1.1108e-08 & 1.4884e-09 & 1.7959e-10 & 5 & $\approx$ 4 \\
     \hline
    pErr@alp   &  9.7359e-10 & 8.6795e-11 & 1.2635e-11 & 1.4572e-12 & 10 &  $\approx$ 4 \\
    \hline
    \end{tabular}
    \caption{Maximum error at the nodes  and the interface point of approximate pressure }
      \label{fig:P2press}
\end{table}


\begin{table}[h]
\centering
    \begin{tabular}{|c||c|c|c|c|c|c|}
    \hline
    Problem 2 & h=1/16 & h=1/32 & h=1/64 & h=1/128 & m & order  \\
    \hline\hline
    uErrEndNodes & 2.7964e-08 & 1.7779e-09 & 1.1119e-10 & 6.9698e-12 & 2 & $\approx$ 4 \\
    \hline
    uErrEndNodes & 3.3232e-08 & 2.0842e-09 & 1.3032e-10 & 8.1454e-12 & 5 & $\approx$ 4 \\
    \hline
    uErrEndNodes & 3.4227e-08 & 2.1550e-09 & 1.3493e-10 & 8.4356e-12 & 10 & $\approx$ 4 \\
    \hline
 uErr@alp   & 3.0707e-08 & 1.9893e-09 & 1.2439e-10 & 7.8315e-12  & 2 & $\approx$ 4 \\
 \hline
  uErr@alp   & 4.2560e-08 & 2.6929e-09 & 1.6839e-10 &  1.0545e-11   & 5 & $\approx$ 4 \\
  \hline
   uErr@alp   &  4.5229e-08 &2.8495e-09 & 1.7841e-10 & 1.1155e-11 & 10 & $\approx$ 4 \\
   \hline
    \end{tabular}
    \caption{Maximum error at the nodes and the interface point $\alpha$ of approximate flux }
      \label{fig:P2flx}
\end{table}

\end{document}